\documentclass[12pt,reqno]{amsart}

\newcommand\version{June 8, 2009}

\usepackage{amsmath,amsfonts,amsthm,amssymb,amsxtra}

\newtheorem{theorem}{Theorem}[section]
\newtheorem{proposition}[theorem]{Proposition}
\newtheorem{lemma}[theorem]{Lemma}

\theoremstyle{definition}

\theoremstyle{remark}

\newtheorem{remark}[theorem]{Remark}

\numberwithin{equation}{section}

\newcommand{\C}{\mathbb{C}}

\renewcommand{\epsilon}{\varepsilon}

\renewcommand{\H}{\mathbb{H}}
\newcommand{\loc}{{\rm loc}}

\newcommand{\N}{\mathbb{N}}

\renewcommand{\phi}{\varphi}
\newcommand{\R}{\mathbb{R}}

\newcommand{\Sph}{\mathbb{S}}

\DeclareMathOperator{\re}{Re}
\DeclareMathOperator{\Span}{span}

\begin{document}

\title[Eigenvalues on the Heisenberg group --- \version]{Inequalities between Dirichlet and Neumann eigenvalues on the Heisenberg group}

\author{Rupert L. Frank}
\address{Rupert L. Frank, Department of Mathematics,
Princeton University, Washington Road, Princeton, NJ 08544, USA}
\email{rlfrank@math.princeton.edu}

\author[A. Laptev]{Ari Laptev}
\address{Ari Laptev, Department of Mathematics, Imperial College London, London SW7 2AZ, UK 
$\&$ Department of Mathematics, Royal Institute of Technology, 100 44 Stockholm, Sweden}
\email{a.laptev@imperial.ac.uk $\&$ laptev@math.kth.se}

\thanks{\copyright\, 2009 by the authors. This paper may be reproduced, in its entirety, for non-commercial purposes.}

\begin{abstract}
We prove that for any domain in the Heisenberg group the $(k+1)$'th Neumann eigenvalue of the sub-Laplacian is strictly less than the $k$'th Dirichlet eigenvalue. As a byproduct we obtain similar inequalities for the Euclidean Laplacian with a homogeneous magnetic field.
\end{abstract}

\maketitle

\section{Introduction and main result}

Universal eigenvalue inequalities are a classical topic in the spectral theory of differential operators. Most relevant to our work here are comparison theorems between the Dirichlet and Neumann eigenvalues $\lambda_j(-\Delta_\Omega^D)$ and $\lambda_j(-\Delta_\Omega^N)$, $j\in\N$, of the Laplacian in a smooth, bounded domain $\Omega\subset\R^d$. Note that $\lambda_j(-\Delta_\Omega^N)\leq \lambda_j(-\Delta_\Omega^D)$ for all $j\in\N$ by the variational characterization of eigenvalues. This trivial bound for $j=1$ was strengthened by P\'olya \cite{Pol} who observed that $\lambda_2(-\Delta_\Omega^N)<\lambda_1(-\Delta_\Omega^D)$ for $d=2$. Payne \cite{Pay}, Aviles \cite{Avi} and Levine and Weinberger \cite{LevWei} obtained further results in this direction under suitable convexity assumptions on $\Omega$. A breakthrough was made by Friedlander \cite{Fri} who proved that 
\begin{equation}\label{eq:fried}
\lambda_{j+1}(-\Delta_\Omega^N)\leq \lambda_j(-\Delta_\Omega^D) 
\qquad\text{for all}\ j\in\N \,,
\end{equation}
\emph{without} any curvature assumption on $\partial\Omega$. Later, Filonov \cite{Fi} simplified Friedlander's proof, removed the smoothness assumption on $\partial\Omega$ and showed that \eqref{eq:fried} is strict for $d\geq 2$. While it is still open whether the Payne--Levine--Weinberger bound $\lambda_{j+d}(-\Delta_\Omega^N)\leq \lambda_j(-\Delta_\Omega^D)$ holds for non-convex domains in $\R^d$, the attention has recently shifted to non-Euclidean analogues of \eqref{eq:fried}. Mazzeo \cite{Maz} has shown for instance that \eqref{eq:fried} holds for domains in hyperbolic space but may fail for domains on the sphere; see also \cite{AshLev} and \cite{HsuWan}.

Our goal in this paper is to obtain the analogue of \eqref{eq:fried} on the Heisenberg group. In this setting \eqref{eq:fried} was previously known only under rather restrictive and non-generic geometric assumptions on $\Omega$. Here we will manage to remove these conditions and, as a bonus, obtain similar inequalities for the Euclidean Laplacian with a homogeneous magnetic field.

The Heisenberg group $\H$ is the prime example of non-commutative harmonic analysis and we refer to \cite{Ste} for background material. We consider $\H$ as $\R^3$ with coordinates $(x,y,t)$ and the (non-commutative) multiplication $(x,y,t)\circ (x',y',t')= (x+x',y+y',t+t'-2(xy'-yx'))$. The vector fields
$$
X= \frac{\partial}{\partial x} + 2y\frac{\partial}{\partial t} \,,
\qquad
Y= \frac{\partial}{\partial y} - 2x\frac{\partial}{\partial t}
$$ 
are left-invariant and the sub-Laplacian on $\H$ is given by
$$
-X^2-Y^2 = -\left(\frac{\partial}{\partial x} + 2y\frac{\partial}{\partial t}\right)^2- \left( \frac{\partial}{\partial y} - 2x\frac{\partial}{\partial t}\right)^2 \,.
$$
We are interested in the Dirichlet and Neumann realizations of this sub-Laplacian on domains $\Omega\subset\H$. The space $L_2(\Omega)$ is defined with respect to the restriction to $\Omega$ of the Lebesgue measure (which coincides with the Haar measure on $\H$) and hence coincides with its Euclidean counterpart. If $\Omega$ is understood, we denote the norm of $u\in L_2(\Omega)$ simply by $\|u\|$. The Sobolev spaces on the Heisenberg group (in this context also known as Folland--Stein spaces) are defined as follows. We denote by $S^1(\Omega)$ the space of all $u\in L_2(\Omega)$ for which the distributional derivatives $Xu$ and $Yu$ belong to $L_2(\Omega)$, equipped with the norm $(\|Xu\|^2+\|Yu\|^2+\|u\|^2)^{1/2}$. The space \textit{\r{S}}$^1(\Omega)$ is defined as the closure of $C_0^\infty(\Omega)$ in $S^1(\Omega)$. The Dirichlet and the Neumann sub-Laplacians $L_\Omega^D$ and $L_\Omega^N$ on $\Omega$ are defined as the self-adjoint operators in $L_2(\Omega)$ corresponding to the quadratic form
$$
\| X u \|^2 + \|Y u \|^2 = \int_\Omega \left( \left|X u\right|^2 + \left|Y u\right|^2 \right) \,dx\,dy\,dt
$$
with form domains \textit{\r{S}}$^1(\Omega)$ and $S^1(\Omega)$, respectively. For any lower semi-bounded operator $A$ with purely discrete spectrum (which is equivalent to its form domain being compactly embedded into the underlying Hilbert space) we denote by $\lambda_j(A)$, $j\in\N$, the $j$-th eigenvalue of $A$, counting multiplicities. The variational principle implies immediately the inequality $\lambda_j(L_\Omega^N)\leq \lambda_j(L_\Omega^D)$ for all $j$. Our main result is the analogue of Friedlander's inequality \eqref{eq:fried} on $\H$. We shall prove

\begin{theorem}
 \label{main}
Let $\Omega\subset\H$ be a domain of finite measure such that the embedding $S^1(\Omega)\subset L_2(\Omega)$ is compact. Then $\lambda_{j+1}(L_\Omega^N)<\lambda_j(L_\Omega^D)$ for any $j\in\N$.
\end{theorem}

\begin{remark}
The assumption that the embedding $S^1(\Omega)\subset L_2(\Omega)$ is compact can be relaxed. Indeed, our proof shows that if $\Omega\subset\H$ is a domain of finite measure (which implies that $L_\Omega^D$ has discrete spectrum) then the total spectral multiplicity of the operator $L_\Omega^N$ in the interval $[0,\lambda_j(L_\Omega^D))$ is at least $j+1$.
\end{remark}

Theorem \ref{main} holds also on the higher-dimensional Heisenberg groups $\H^{n}$; see Section~\ref{sec:ext}.

We close this introduction by commenting on the similarities and differences between the proofs of \eqref{eq:fried} in the Heisenberg and in the Euclidean case. As emphasized by Mazzeo \cite{Maz}, Friedlander's proof of the Euclidean inequality \eqref{eq:fried} relies on the existence, for any $\lambda>0$, of a function $U$ such that
\begin{equation}\label{eq:expon}
-\Delta U = \lambda U
\quad\text{and}\quad
|\nabla U| \leq \sqrt\lambda |U| \,.
\end{equation}
Of course, on Euclidean space such functions are provided by $U(x)=e^{i\sqrt\lambda x\cdot\omega}$, $\omega\in\Sph^{d-1}$. Actually, an inspection of the proofs in \cite{Fri,Fi} shows that the second, pointwise property in \eqref{eq:expon} can be relaxed to the averaged property
$$
\int_\Omega |\nabla U|^2\,dx \leq \lambda\int_\Omega |U|^2 \,dx \,.
$$
Similarly, we will prove Theorem \ref{main} by constructing functions $U$ such that
\begin{equation}\label{eq:exponh}
-(X^2+Y^2)U = \lambda U
\quad\text{and}\quad
\|X U \|^2_{L_2(\Omega)} + \|Y U \|^2_{L_2(\Omega)} \leq \lambda \|U\|^2_{L_2(\Omega)} \,.
\end{equation}
This construction is described in Subsection \ref{sec:extrial} and constitutes the main novelty of this paper. While it is easy to find explicit solutions $U_{z'}$ of the equation in \eqref{eq:exponh}, depending on a parameter $z'\in\R^2$, it seems rather difficult to prove that for given $z'$ and $\Omega$ the inequality in \eqref{eq:exponh} is satisfied. Our way around this impasse is to show that the energy inequality holds after \emph{averaging} over $z'\in\R^2$. We believe that this averaging technique might have further applications beyond the present context.

For the sake of clarity we carry out the averaging procedure first for the two-dimensional Landau operator.
We emphasize that the connection between this operator and the sub-Laplacian on the Heisenberg group was also essential in the recent proof of sharp Berezin--Li--Yau inequalities on $\H$ \cite{HanLap}; see also \cite{Str}. Eigenvalue inequalities for the Landau operator which we obtain along our way to Theorem \ref{main} are presented in the final Section \ref{sec:ext}.

\subsection*{Acknowledgements}
The authors acknowledge interesting discussions with A. Hansson concerning the topics of this paper. The first author wishes to thank E. Lieb and R. Seiringer for helpful remarks. Support through DFG grant FR 2664/1-1 and U.S. NSF grant PHY 06 52854  (R.F.) is gratefully acknowledged.


\section{Proof of Theorem \ref{main}}\label{sec:proof}

\subsection{Eigenfunctions of the two-dimensional Landau operator}\label{sec:extrial}

For $z=(x,y)\in\R^2$ let $\mathbf A(x,y):=\frac12 (-y,x)^T$ and $\mathbf D=-i\nabla$. For $B>0$ the spectrum of the self-adjoint operator $(\mathbf D- B\mathbf A)^2$ in $L_2(\R^2)$ consists of the points $B(2k-1)$, $k\in\N$, each being an eigenvalue of infinite multiplicity. Hence there exist infinitely many linearly independent functions $U$ on $\R^2$ satisfying $(\mathbf D- B\mathbf A)^2 U = B(2k-1) U$ and $\int_{\R^2} |(\mathbf D- B\mathbf A) U|^2 \,dz = B(2k-1) \int_{\R^2} |U|^2 \,dz$. It is a non-trivial question, however, whether for a given domain $\Omega$ one can find $U$'s such that $\int_\Omega | (\mathbf D- B\mathbf A) U|^2 \,dz \leq B(2k-1) \int_\Omega |U|^2 \,dz$. That the answer is affirmative is the content of

\begin{proposition}\label{extrial}
 Let $B>0$, $k\in\N$ and $\Omega\subset\R^2$ a domain of finite measure. There are infinitely many linearly independent functions $U \in C^\infty(\Omega)\cap L_2(\Omega)$ satisfying
\begin{align*}
(\mathbf D- B\mathbf A)^2 U & = B(2k-1) U \qquad \text{in} \ \Omega\,, \\
\int_\Omega | (\mathbf D- B\mathbf A) U|^2 \,dz & \leq B(2k-1) \int_\Omega |U|^2 \,dz \,.
\end{align*}
\end{proposition}

In order to prove this proposition we use some properties of the spectral projection $P_k^B$ corresponding to the eigenvalue $B(2k-1)$, $k\in\N$, of the operator $(\mathbf D- B\mathbf A)^2$ in $L_2(\R^2)$. This projection is an integral operator with integral kernel
\begin{equation}
\label{eq:explicit}
P_k^B(z,z') = \frac{B}{2\pi} e^{-i B z\times z'/2 - B|z-z'|^2/4} L_{k-1}(B|z-z'|^2/2) \,.
\end{equation}
Here $L_{k-1}$ denotes the Laguerre polynomial of degree $k-1$, normalized by $L_{k-1}(0)=1$. We will choose the $U$'s in Theorem \ref{extrial} as $P_k^B(\cdot,z')$ for different values of $z'$. Indeed, since $P_k^B$ is a projector corresponding to $B(2k-1)$, one has 
\begin{equation}\label{eq:eq}
(\mathbf D_z- B\mathbf A(z))^2 P^B_k(z,z')=B(2k-1)P^B_k(z,z')
\end{equation}
for any $z'$. In order to find $z'$'s for which the claimed energy bound holds we use the following averaging lemma. It appeared in \cite{Fr} in a different context and we include here a proof for the sake of completeness.

\begin{lemma}\label{mainlemma}
 Let $B>0$ and $k\in\N$. Then for all $z\in\R^2$
\begin{equation}\label{eq:homneuproof}
\int_{\R^2} |(\mathbf D_z-B\mathbf A(z)) P^B_k(z,z')|^2 \,dz' = B(2k-1) \int_{\R^2} |P^B_k(z,z')|^2 \,dz' \,.
\end{equation}
\end{lemma}

We emphasize that the integration in \eqref{eq:homneuproof} is with respect to the variable $z'$. The identity is also true (and easier to prove) when the integrals are performed with respect to $z$ with $z'$ fixed. Our proof below does not use the explicit form \eqref{eq:explicit}, but only that $P^B_k$ is smooth and is constant on the diagonal (which follows by the magnetic translation covariance of the Landau operator).

\begin{proof}
We denote $\mathbf Q_z := \mathbf D_z-B \mathbf A(z)$ and abbreviate $P:=P_k^B$. Since $P^2=P$, the left side of \eqref{eq:homneuproof} equals $\mathbf Q_z \overline{\mathbf{ Q}_{z'}} P(z,z')|_{z=z'}$. Using this and that $P(x,x)=B/2\pi$, the right side equals $B^2(2k-1)/(2\pi)$. By \eqref{eq:eq} one has
$$\mathbf Q_z^2 P(z,z')|_{z=z'} = B(2k-1)\frac B{2\pi}
\qquad \text{and} \qquad 
\overline{\mathbf Q_{z'}}^2 P(z,z')|_{z=z'} = B(2k-1)\frac B{2\pi}\,,
$$
and therefore it suffices to prove that
\begin{equation}\label{eq:kernelder}
\left(\mathbf Q_z^2+ \overline{\mathbf Q_{z'}}^2 - 2 \mathbf Q_z \overline{\mathbf{ Q}_{z'}} \right) P(z,z')|_{z=z'} = 0 \,.
\end{equation}
Now we expand $\mathbf Q_z$ and $\mathbf Q_{z'}$ and write $\mathbf Q_z^2+ \overline{\mathbf Q_{z'}}^2 - 2 \mathbf Q_z \overline{\mathbf{ Q}_{z'}}$ as a sum of three terms, containing only derivatives of order zero, one and two, respectively. The zeroth order term is easily seen to vanish if $z=z'$. The first order term is given by $- 2B \left(\mathbf A(z)-\mathbf A(z')\right)\cdot \left(\mathbf{D}_z + \mathbf{D}_{z'} \right)$ and hence also vanishes if $z=z'$. Thus \eqref{eq:kernelder} is equivalent to
$$
\left(\mathbf D_z^2+ \mathbf D_{z'}^2 + 2 \mathbf D_z \mathbf{ D}_{z'} \right) P(z,z')|_{z=z'}
=0 \,.
$$
The latter equality follows by differentiating the identity $P_k(z,z)=B/2\pi$ twice with respect to $z$. This concludes the proof of \eqref{eq:homneuproof}.
\end{proof}

We now turn to the

\begin{proof}[Proof of Proposition \ref{extrial}]
Recalling \eqref{eq:eq} we will look for $U$ in the form $P_k^B(\cdot,z')$. According to Lemma \ref{mainlemma},
$$
\int_{\R^2} \int_\Omega |(\mathbf D_z-B\mathbf A(z)) P_k^B(z,z')|^2 \,dz \,dz' 
= B(2k-1) \int_{\R^2} \int_\Omega |P_k^B(z,z')|^2 \,dz \,dz' \,.
$$
As observed in the proof of that lemma the right hand side equals $B(2k-1) \frac{B}{2\pi} |\Omega|$ and hence both sides are finite. Hence the set $K$ of all $z'\in\R^2$ such that
\begin{equation}\label{eq:defk}
\int_\Omega |(\mathbf D_z-B\mathbf A(z)) P_k^B(z,z') |^2 \,dz \leq B(2k-1) \int_\Omega |P_k^B(z,z') |^2 \,dz
\end{equation}
has positive measure. To complete the proof we have to show that the set $\{ \chi_\Omega P_k^B(\cdot,z') :\ z'\in K\}$ is infinite dimensional. 

By Fubini's theorem there is an $a\in\R$ such that $\Gamma:=\{x'\in\R : (x',a)\in K\}$ has positive measure. Let $b\in\R$ such that $I:=\{ x\in\R : (x,b)\in\Omega\}$ is non-empty. We claim that the functions $P_k^B((\cdot,b),z')$, $z'\in\Gamma$, are linearly independent on $I$. Indeed, if
$$
\sum_{j=1}^N \alpha_j P_k^B((x,b),w^{(j)}) = 0
\qquad\text{for all}\ x\in I
$$
and some $\alpha_j\in\C$ and $w^{(j)}=(s^{(j)},a)\in\Gamma$, then by \eqref{eq:explicit}
$$
\sum_{j=1}^N \tilde\alpha_j e^{B(x s^{(j)}-ix a)/2} L_{k-1}(B((x-s^{(j)})^2+ (a-b)^2)/2) = 0
\qquad\text{for all}\ x\in I\,,
$$
where $\tilde\alpha_j:= e^{iBbs^{(j)}/2-B(s^{(j)})^2/4}\alpha_j$. Since the left-hand side of this identity is a real-analytic function of $x$, it holds for all $x\in\R$. Letting $x\to\infty$ one easily concludes that $\tilde\alpha_j=0$ for all $j$, and hence also $\alpha_j=0$, as claimed.
\end{proof}

\begin{remark}\label{extrialrem}
 Proposition \ref{extrial} has a three-dimensional analogue. Indeed, the same proof shows that if $B>0$, $k\in\N$ and $\Omega\subset\R^3$ is a domain of finite measure there exist infinitely many linearly independent functions $U \in C^\infty(\Omega)\cap L_2(\Omega)$, depending only on the variables $(x,y)\in\R^2$, such that
\begin{align*}
(\mathbf D_{(x,y)}- B\mathbf A(x,y))^2 U & = B(2k-1) U \qquad \text{in} \ \Omega\,, \\
\int_\Omega | (\mathbf D_{(x,y)}- B\mathbf A(x,y)) U|^2 \,dx\,dy\,dt & \leq B(2k-1) \int_\Omega |U|^2 \,dx\,dy\,dt \,.
\end{align*}
\end{remark}


\subsection{Proof of Theorem \ref{main}}

Given Remark \ref{extrialrem}, Theorem \ref{main} follows similarly as in \cite{Fi}. We include the proof not only in order to make this presentation self-contained, but also since we have managed to simplify Filonov's proof by avoiding the use of a unique continuation result.

We abbreviate $\lambda_j^D := \lambda_j(L_\Omega^D)$ and similarly for the Neumann eigenvalues. Let $j\in\N$ be fixed and denote by $\phi_1^D,\ldots,\phi_j^D$ orthonormal eigenfunctions corresponding to the eigenvalues $\lambda_1^D,\ldots,\lambda_j^D$. Moreover, we choose $k\in\N$ and $\tau>0$ such that $4\tau(2k-1)= \lambda_j^D$. According to Remark \ref{extrialrem} there exists a smooth function $U$ on $\Omega$ depending only on the variables $(x,y)$ such that
\begin{equation}
\label{eq:u}
\begin{split}
(\mathbf D_{(x,y)}- 4\tau\mathbf A(x,y))^2 U & = 4\tau(2k-1) U \qquad \text{in} \ \Omega\,, \\
\int_\Omega | (\mathbf D_{(x,y)}- 4\tau\mathbf A(x,y)) U|^2 \,dx\,dy\,dt & \leq 4\tau(2k-1) \int_\Omega |U|^2 \,dx\,dy\,dt \,.
\end{split}
\end{equation}
and such that $e^{i\tau t} U$ is linearly independent of $\phi_1^D,\ldots,\phi_j^D$ and of the space $\mathcal N$ spanned by all Neumann eigenfunctions corresponding to eigenvalues less or equal to $\lambda_{j+1}^N$. (We emphasize that if $\lambda_{j+1}^N$ is degenerate, the dimension of $\mathcal N$ might exceed $j+1$, but is finite by the compactness assumption.) With this choice of $U$ the space
$$
\mathcal M :=  \Span \{\phi_1^D,\ldots,\phi_j^D,e^{i\tau t} U\}
$$
is $j+1$-dimensional and hence by the variational principle
\begin{equation}\label{eq:varprinc}
\lambda_{j+1}^N \leq \sup_{0\not\equiv u\in\mathcal M} \frac{ \|Xu\|^2+\|Yu\|^2}{\|u\|^2} \,.
\end{equation}
In order to estimate the Rayleigh quotient we write an arbitrary $u\in\mathcal M$ as
$$
u(x,y,t) := \sum_{i=1}^j \alpha_i \phi_i^D(x,y,t) + \alpha_{j+1} e^{i\tau t} U(x,y)
$$
with constants $\alpha_1,\ldots,\alpha_{j+1}\in\C$. Using the equation of the $\phi_i^D$ and their orthogonality we obtain
\begin{align*}
\|Xu\|^2+\|Yu\|^2
= & \sum_{i=1}^j \lambda_i^D |\alpha_i|^2 
+ |\alpha_{j+1}|^2 \int_\Omega \left( \left|X e^{i\tau t} U\right|^2 + \left|Y e^{i\tau t} U\right|^2 \right) \,dx\,dy\,dt \\
& + 2\re \sum_{i=1}^j \overline{\alpha_{j+1}}\alpha_i \int_\Omega \left( \overline{X e^{i\tau t} U} X\phi_i^D + \overline{Y e^{i\tau t} U} Y\phi_i^D \right) \,dx\,dy\,dt \,.
\end{align*}
Note that $(X e^{i\tau t} U, Y e^{i\tau t} U)^T = i e^{i\tau t} (\mathbf D_{(x,y)}- 4\tau\mathbf A(x,y)) U$. Integrating by parts, using that $\phi_i^D$ satisfies Dirichlet boundary conditions and recalling the equation in \eqref{eq:u} for $U$ yields
\begin{align*}
& \int_\Omega \left( \overline{X e^{i\tau t} U} X\phi_i^D + \overline{Y e^{i\tau t} U} Y\phi_i^D \right) \,dx\,dy\,dt \\
& \quad = \int_\Omega e^{-i\tau t} \overline{ (\mathbf D_{(x,y)}- 4\tau\mathbf A(x,y))^2 U} \ \phi_i^D \,dx\,dy\,dt
= 4\tau(2k-1) \int_\Omega e^{-i\tau t} \overline{U} \phi_i^D \,dx\,dy\,dt \,.
\end{align*}
Moreover, by the estimate in \eqref{eq:u}
\begin{align*}
& \int_\Omega \left( \left|X e^{i\tau t} U\right|^2 + \left|Y e^{i\tau t} U\right|^2 \right) \,dx\,dy\,dt \\
& \quad = \int_\Omega | (\mathbf D_{(x,y)}- 4\tau\mathbf A(x,y)) U|^2 \,dx\,dy\,dt
\leq 4\tau(2k-1) \int_\Omega |U(x,y)|^2 \,dx\,dy\,dt \,.
\end{align*}
Hence, estimating $\lambda_i^D\leq\lambda_j^D$ and recalling that $4\tau(2k-1)=\lambda_j^D$ we obtain
\begin{align*}
\|Xu\|^2+\|Yu\|^2
& \leq \lambda_j^D \left( \sum_{i=1}^j |\alpha_i|^2 + |\alpha_{j+1}|^2 \int_\Omega |U(x,y)|^2 \,dx\,dy\,dt
 \right. \\
& \qquad \qquad \left. + 2 \re \sum_{i=1}^j \overline{\alpha_{j+1}}\alpha_i \int_\Omega e^{-i\tau t} \overline{U} \phi_i^D \,dx\,dy\,dt \right) \\
& = \lambda_j^D \|u\|^2 \,.
\end{align*}
By the variational principle, see \eqref{eq:varprinc}, this implies that $\lambda_{j+1}^N\leq \lambda_j^D$. Moreover, the inequality in \eqref{eq:varprinc} is strict unless $\mathcal M\subset\mathcal N$. But this is impossible since we have chosen $e^{i\tau t}U$ to be linearly independent of $\mathcal N$. This proves Theorem \ref{main}.


\section{Two extensions}\label{sec:ext}

\subsection{The Landau operator}

In this subsection we let $d=2$ or $d=3$. If $d=2$ we use coordinates $z=(x,y)$ and define $\mathbf A(x,y):=\frac12 (-y,x)^T$. If $d=3$ we use coordinates $z=(x,y,t)$ and define $\mathbf A(x,y,t):=\frac12 (-y,x,0)^T$. For a domain $\Omega\subset\R^d$ we put $H^1_{B \mathbf A}(\Omega):=\{u\in L_2(\Omega)\cap H^1_\loc(\Omega):\ (\mathbf D- B \mathbf A)u \in L_2(\Omega)\}$ with norm $\left(\|(\mathbf D-B\mathbf A)u\|^2+\|u\|^2\right)^{1/2}$ and denote by \r{H}$^1_{B \mathbf A}(\Omega)$ the closure of $C_0^\infty(\Omega)$ in $H^1_{B\mathbf A}(\Omega)$. The self-adjoint operators $H_\Omega^D(B)$ and $H_\Omega^N(B)$ in $L_2(\Omega)$ are defined via the quadratic forms
$$
\|(\mathbf D-B\mathbf A)u\|^2 = \int_\Omega |(\mathbf D- B\mathbf A)u|^2 \,dz
$$
with form domains \r{H}$^1_{B \mathbf A}(\Omega)$ and $H^1_{B \mathbf A}(\Omega)$, respectively.

\begin{theorem}\label{landau}
Let $B>0$ and let $\Omega\subset\R^d$, $d=2,3$, be a domain of finite measure such that the embedding $H^1_{B\mathbf A}(\Omega)\subset L_2(\Omega)$ is compact. 
\begin{enumerate}
\item
If $d=2$ let $k\in\N$ and assume that $H_\Omega^D(B)$ has $j$ eigenvalues less or equal to $B(2k-1)$. Then $H_\Omega^N(B)$ has $j+1$ eigenvalues less than $B(2k-1)$.
\item
If $d=3$ then $\lambda_{j+1}(H_\Omega^N(B))<\lambda_{j}(H_\Omega^D(B))$ for all $j\in\N$.
\end{enumerate}
\end{theorem}

Note that by the diamagnetic inequality $|(\mathbf D-B\mathbf A)u| \geq |\nabla |u||$ the compactness of $H^1(\Omega)\subset L_2(\Omega)$ is sufficient for the compactness of $H^1_{B\mathbf A}(\Omega)\subset L_2(\Omega)$.

\begin{proof}
First assume that $d=2$. Let $\phi_1^D,\ldots,\phi_j^D$ be the Dirichlet eigenfunctions corresponding to the eigenvalues less or equal $B(2k-1)$ and let $\mathcal N$ be the subspace generated by the Neumann eigenfunctions corresponding to the eigenvalues less or equal $B(2k-1)$. Let $U$ be a function as in Proposition \ref{extrial} which is linearly independent of $\phi_1^D,\ldots,\phi_j^D$ and $\mathcal N$. Then any function $u$ in the span of $\phi_1^D,\ldots,\phi_j^D$ and $U$ satisfies by a similar calculation as in the proof of Theorem \ref{main}
$$
\int_\Omega |(\mathbf D- B\mathbf A)u|^2 \,dx\,dy \leq B(2k-1) \int_\Omega |u|^2 \,dx\,dy \,.
$$
Hence the $(j+1)$'th Neumann eigenvalue is less or equal to $B(2k-1)$, and equality is excluded as before by linear independence of $U$ and $\mathcal N$.

Now let $d=3$. Let $\phi_1^D,\ldots,\phi_j^D$ be Dirichlet eigenfunctions corresponding to the eigenvalues $\lambda_i^D:=\lambda_{i}(H_\Omega^D(B))$, $i=1,\ldots,j$, and let $\mathcal N$ be the subspace generated by the Neumann eigenfunctions corresponding to the eigenvalues less or equal $\lambda_j^N:=\lambda_{i}(H_\Omega^N(B))$. Since $\lambda_j^D\geq \lambda_1^D\geq B$ we can choose $k\in\N$ and $\tau\in\R$ such that $B(2k-1)+\tau^2=\lambda_j^D$. Let $U$ be a function as in Remark \ref{extrialrem} such that $e^{i\tau t} U$ is linearly independent of $\phi_1^D,\ldots,\phi_j^D$ and $\mathcal N$. Using that
\begin{align*}
\int_\Omega |(\mathbf D_z- B\mathbf A(z))e^{i\tau t} U|^2 \,dz
& = \int_\Omega \left(|(\mathbf D_{(x,y)}- B\mathbf A(x,y)) U|^2 + \tau^2 |U|^2\right) \,dz \\
& \leq \left(B(2k-1)+\tau^2\right) \int_\Omega |e^{i\tau t} U|^2 \,dz
\end{align*}
one finds that any function $u$ in the span of $\phi_1^D,\ldots,\phi_j^D$ and $e^{i\tau t}U$ satisfies
$$
\int_\Omega |(\mathbf D- B\mathbf A)u|^2 \,dz \leq \lambda_j^D \int_\Omega |u|^2 \,dz
$$
and one derives the asserted inequality as before.
\end{proof}

\begin{remark}
 Similarly as in the non-magnetic case Theorem \ref{landau} remains valid if instead of Neumann boundary conditions one considers Robin boundary conditions with a density having non-positive average. To be more precise, assume that the trace embedding $H^1_{B\mathbf A}(\Omega)\subset L_2(\partial\Omega)$ is compact and let $\sigma\in L_\infty(\partial\Omega)$ be a real-valued function with $\int_{\partial\Omega} \sigma \,d\omega(z)\leq 0$. (Here $d\omega$ denotes the surface measure on $\partial\Omega$.) The self-adjoint operator $H_\Omega^{(\sigma)}(B)$ is defined via the quadratic form
$$
\int_\Omega |(\mathbf D- B\mathbf A)u|^2 \,dz + \int_{\partial\Omega} \sigma |u|^2 \,d\omega(z)
$$
with form domain $H^1_{B \mathbf A}(\Omega)$. Then Theorem \ref{landau} remains valid with $H_\Omega^N(B)$ replaced by $H_\Omega^{(\sigma)}(B)$. The proof is based on an analogue of Proposition \ref{extrial} with the energy bound
\begin{align*}
\int_\Omega | (\mathbf D- B\mathbf A) U|^2 \,dz + \int_{\partial\Omega} \sigma |U|^2 \,d\omega(z)
 \leq B(2k-1) \int_\Omega |U|^2 \,dz \,.
\end{align*}
The existence of such $U$'s is proved as before by averaging using the fact that
$$
\int_{\R^2} \int_{\partial\Omega} \sigma(z) |P^B_k(z,z')|^2 \,d\omega(z) \,dz' = \frac{B}{2\pi} \int_{\partial\Omega} \sigma(z) \,d\omega(z)\leq 0 \,.
$$
\end{remark}

\begin{remark}\label{landaumulti}
 Theorem \ref{landau} has generalizations to dimensions $d\geq 4$. In $\R^d$ we use coordinates $(x_1,y_1,\ldots,x_n,y_n,t_1,\ldots,t_m)$ where $2n+m=d$. If $d$ is even we allow $m=0$. For $\mathbf B=(B_1,\ldots,B_n)$ with all $B_j>0$ let
\begin{equation}
 \label{eq:amulti}
\mathbf A_\mathbf B(x_1,y_1,\ldots,x_n,y_n):=\tfrac 12 (-B_1y_1,B_1x_1,\ldots,-B_ny_n,B_nx_n,0,\ldots,0)^T \,.
\end{equation}
If $\Omega\subset\R^d$ is a domain we define the operators $H_\Omega(\mathbf B)$ and $H_\Omega^N(\mathbf B)$ similarly as before. We claim that if $m=0$ then the analogue of the first part of Theorem \ref{landau} is valid (with $B(2k-1)$ replaced by $\sum_{j=1}^n B_j(2k_j-1)$ for $k_j\in\N$), whereas if $m\geq 1$ then the analogue of the second part of Theorem \ref{landau} is valid. This follows by a similar argument as before, but now we consider functions $P_{k_1}^{B_1}(z_1,z_1')\cdots P_{k_n}^{B_n}(z_n,z_n')$ with $z_j=(x_j,y_j)$ and we average over the parameters $z_j'$, $j=1,\ldots,n$.
\end{remark}


\subsection{Higher dimensional Heisenberg groups}

Finally, we prove a generalization of Theorem \ref{main} to the higher dimensional Heisenberg groups $\H^n$, that is, $\R^{2n+1}$ with coordinates $(x_1,y_1,\ldots,x_n,y_n,t)$ and multiplication
\begin{align*}
& (x_1,y_1,\ldots,x_n,y_n,t)\circ (x_1',y_1',\ldots,x_n',y_n',t')\\
&\qquad = \left(x_1+x_1',y_1+y_1',\ldots,x_n+x_n',,t+t'-2\sum_j(x_jy_j'-y_jx_j')\right) \,.
\end{align*}
The vector fields
$$
X_j= \frac{\partial}{\partial x_j} + 2y_j\frac{\partial}{\partial t} \,,
\qquad
Y_j= \frac{\partial}{\partial y_j} - 2x_j\frac{\partial}{\partial t}
$$
are left-invariant and the sub-Laplacian on $\H^n$ is given by $-\sum_{j=1}^n \left(X_j^2+Y_j^2\right)$. If $\Omega\subset\H^n$ is a domain then the Sobolev spaces $S^1(\Omega)$ and \textit{\r{S}}$^1(\Omega)$ are defined similarly as for $n=1$ and the Dirichlet and Neumann sub-Laplacians $L_\Omega^D$ and $L_\Omega^N$ are defined through the quadratic form
$$
\sum_{j=1}^n \left( \| X_j u \|_{L_2(\Omega)}^2 + \|Y_j u \|_{L_2(\Omega)}^2 \right)
$$
with form domains \textit{\r{S}}$^1(\Omega)$ and $S^1(\Omega)$, respectively. We shall prove that the eigenvalues of these operators satisfy the same bound as in the case $n=1$.

\begin{theorem}
 Let $\Omega\subset\H^n$ be a domain of finite measure such that the embedding $S^1(\Omega)\subset L_2(\Omega)$ is compact. Then $\lambda_{j+1}(L_\Omega^N)<\lambda_j(L_\Omega^D)$ for any $j\in\N$.
\end{theorem}

\begin{proof}
Let $\tau>0$ and define $\mathbf A_\mathbf \tau$ by \eqref{eq:amulti} with $B_1=\ldots=B_n=\tau$. In Remark~\ref{landaumulti} we have outlined how to construct for given $k_1,\ldots,k_n\in\N$ infinitely many linearly independent functions $U$ of the variables $(x_1,y_1,\ldots,x_n,y_n)$ such that
\begin{align*}
(\mathbf D_{(x,y)} - \mathbf A_{4\mathbf \tau} )^2 U & = 4\tau\left(2\sum k_j - n\right) \ U \,, \\
\| (\mathbf D_{(x,y)} - \mathbf A_{4\mathbf \tau} ) U \|_{L_2(\Omega)}^2 & \leq 4\tau\left(2\sum k_j - n\right) \ \|U\|^2_{L_2(\Omega)} \,.
\end{align*}
The assertion now follows as in the the proof of Theorem \ref{main} by taking $e^{i\tau t} U$ as an additional trial function in the variational principle.
\end{proof}


\bibliographystyle{amsalpha}

\end{document}